\documentclass[11pt]{amsart}
\usepackage{amsmath, amscd, amssymb, mathrsfs, mathabx, tikz-cd, amsthm, thmtools, chngcntr}
\usepackage[curve]{xypic}
\usepackage{enumitem}

\DeclareFontFamily{OT1}{rsfs}{}
\DeclareFontShape{OT1}{rsfs}{n}{it}{<-> rsfs10}{}
\DeclareMathAlphabet{\curly}{OT1}{rsfs}{n}{it}

\usepackage{hyperref}

\input{xy}
\xyoption{all}

\newtheorem{Thm}{Theorem}[section]
\newtheorem{Thm*}{Theorem}

\newtheorem{Conj*}{Conjecture}

\newtheorem{Prop}[Thm]{Proposition}
\newtheorem{Def}[Thm]{Definition}
\newtheorem{Def/Thm}[Thm]{Definition/Theorem}

\newtheorem{Lemma}[Thm]{Lemma}

\theoremstyle{definition}
\newtheorem{Rmk}[Thm]{Remark}

\counterwithout{equation}{section}

\newcommand{\ot }{\otimes}

\newcommand{\id}{{\operatorname{id}}}

\newcommand{\rank }{{\mathrm{rank}\,}}
\newcommand{\cA}{{\mathcal{A}}}

\newcommand{\cO}{{\mathcal{O}}}

\newcommand{\cE}{{\mathcal{E}}}

\newcommand{\cU}{{\mathcal{U}}}

\newcommand{\ccD}{{\mathcal{D}}}

\newcommand{\fC}{{\mathfrak{C}}}

\newcommand{\PP }{{\mathbb P}}
\newcommand{\EE }{{\mathbb E}}

\let\wt\widetilde

\newcommand{\lan}{\langle}
\newcommand{\ran}{\rangle}

\newcommand{\T}{\mathsf{T}}

\newcommand{\cB}{\mathcal{B}}

\renewcommand\;{\hspace{.7pt}}
\newcommand\C{\mathbb C}
\newcommand\Q{\mathbb Q}

\newcommand\Z{\mathbb Z}

\newcommand{\LL}{\mathbb{L}}

\renewcommand\t{\mathfrak t}

\renewcommand\({\big(}
\renewcommand\){\big)}
\makeatletter
\newcommand{\so}{\ \ext@arrow 0359\Rightarrowfill@{}{\hspace{3mm}}\ }
\makeatother
\newcommand{\rt}[1]{\xrightarrow{\ #1\ }}
\newcommand\To{\longrightarrow}

\newcommand\into{\hookrightarrow}

\newcommand\INTO{\ \ar@{^(->}[r]<-.2ex>}

\newcommand\Mapsto{\ \longmapsto\ }
\newcommand\take{\smallsetminus}

\renewcommand\={\ =\ }

\DeclareMathSymbol{\lefttorightarrow}{3}{mathb}{"FC}
\DeclareMathSymbol{\righttoleftarrow}{3}{mathb}{"FD}

\newcommand\rk{\operatorname{rank}}
\newcommand\vir{\operatorname{vir}}

\newcommand\coker{\operatorname{coker}}

\newcommand\Ext{\operatorname{Ext}}

\newcommand\Spec{\operatorname{Spec}}

\newcommand\Cone{\operatorname{Cone}}

\newcommand\Sym{\operatorname{Sym}}

\newcommand\E{\overline{E}}

\newcommand\bcB{\overline{\cB}}

\newcommand\bN{\overline{N}}

\newcommand\bsigma{\overline{\sigma}}
\newcommand\btau{\overline{\tau}}

\newcommand\red{\mathrm{red}}

\newcommand\beq[1]{\begin{equation}\label{#1}}
\newcommand\eeq{\end{equation}}
\newcommand\beqa{\begin{eqnarray*}}
\newcommand\eeqa{\end{eqnarray*}}

\newcommand\arXiv[1]{\href{http://arxiv.org/abs/#1}{arXiv:#1}}
\newcommand\mathAG[1]{\href{http://arxiv.org/abs/math/#1}{math.AG/#1}}

\begin{document}

\title[Projective completions and quantum Lefschetz]{\ \vspace{-15mm}\\ Virtual cycles on projective completions and quantum Lefschetz formula\vspace{-3mm}}
\author[J. Oh]{Jeongseok Oh}


\begin{abstract}
For a compact quasi-smooth derived scheme $M$ with $(-1)$-shifted cotangent bundle $N$, there are at least two ways to localise the virtual cycle of $N$ to $M$ via torus and cosection localisations, introduced by Jiang-Thomas \cite{JT}. We produce virtual cycles on both the projective completion $\bN:=\PP(N\oplus\cO_M)$ and projectivisation $\PP(N)$ and show the ones on $\bN$ push down to Jiang-Thomas cycles and the one on $\PP(N)$ computes the difference. 

Using similar ideas we give an expression for the difference of the quintic and $t$-twisted quintic GW invariants of Guo-Janda-Ruan \cite{GJR}. 
\end{abstract}

\maketitle
\vspace{-6mm}
\setcounter{tocdepth}{1}

\section*{Introduction}
Let $N$ be a quasi-projective scheme equipped with the symmetric obstruction theory $\phi: \EE_N \to \LL_N$, where $\LL_N$ is the (truncated) cotangent complex of $N$. Then $\phi$ defines the degree zero virtual cycle
$$
[N]^{\vir}\ \in\ A_{0}\(N\).
$$
When $N$ is compact, we obtain an invariant $\deg [N]^{\vir} \in \Z$\footnote{It is Donaldson-Thomas invariant when $N$ is a moduli space of stable sheaves on Calabi-Yau $3$-fold.}. Even if $N$ is not compact, 
invariants can be defined via localisations\footnote{By a localisation of a class $x \in A_*\(X\)$ to a closed subscheme $i: Y \into X$, we mean a class $y \in A_*\(Y\)$ such that $i_* y = x$.} so long as it is acted on by a torus with a compact fixed locus. We review some localisations studied by Jiang-Thomas \cite{JT} 
in Appendix \ref{app:A}.

For a compact quasi-smooth derived scheme $M$ with its associated $(-1)$-shifted cotangent bundle $N$\footnote{For instance when $M$ is a moduli space of stable sheaves on a surface $S$, $N$ is the moduli space of stable sheaves on the canonical bundle $K_S$.}, $\T:=\C^*$ acts on $N$ fiberwise so that $M$ becomes its fixed locus. As a (classical) scheme $N$ is roughly the dual of obstruction sheaf over $M$. In this case, Jiang-Thomas show there are really only $2$ different localisations \cite{JT} -- torus and cosection localisations.

Letting $\EE_M\to \LL_M$ be the perfect obstruction theory of $M$ so that $N=\Spec\Sym(h^{1}(\EE^\vee_M))$, $t$ be the Euler class of the standard weight $1$ representation $\t$ of $\T$, and $\sigma: h^1(\EE^\vee_N)\cong\Omega_N\to\cO_N$ be the cosection induced by the Euler vector field, these two cycles are\footnote{The equivariant cycle $[M]^{\vir}/e_{\T}\(\EE_{M}[-1]\ot\t\)$ is then a polynomial in $t$ by degree reason. Since it is of degree zero, the coefficient of $t^i$ should be a degree $i$ Chow class.} 
\beq{virs}
[N]^{\vir}_{\T}\ :=\ \left\{\frac{[M]^{\vir}}{e_{\T}\(\EE_{M}[-1]\ot\t\)}\right\}_{t=0} \ \ \text{ and }\ \  [N]^{\vir}_{\sigma} \ \in\ A_0\(M\),
\eeq
where the latter is the cosection localised cycle \cite{KL}. Our main result is the difference of the two is given by the reduced cycle of the projectivisation $\PP(N)$. The composition $\EE_N|_{N\take M}\rt{\phi}\LL_{N\take M}\to\Omega_{(N\take M)/\PP(N)}$ defines a virtual rank $-1$ perfect obstruction theory of $\PP(N)$
$$
\EE_{\PP(N)}\ :=\ \Cone\(\,\EE_N|_{N\take M}\ \To\ \Omega_{(N\take M)/\PP(N)}\)[-1]
$$
in $D_{\C^*}\(N\take M\)\cong D\(\PP(N)\)$. The residue map $\Omega_{\PP(N\oplus\; \cO_M)}\(\log\PP(N)\)\(\PP(N)\)\to\cO_{\PP(N)}\(\PP(N)\)$ factors through
$$
\Omega_{\PP(N\oplus\; \cO_M)}\(\log\PP(N)\)\(\PP(N)\)\ \To\ \cO_{\PP(N\oplus\;\cO_M)}\(\PP(N)\)
$$
extending $\sigma:\Omega_N\to \cO_N$, whose restriction 
to $\PP(N)$ defines a surjective cosection\footnote{A way of thinking $\Omega_{\PP(N\oplus\; \cO_M)}\(\log\PP(N)\)\(\PP(N)\)|_{\PP(N)}$ is isomorphic to the obstruction bundle $h^1\(\EE^\vee_{\PP(N)}\)\cong h^1\(\EE^\vee_N|_{N\take M}\)\cong h^0\(\EE_N|_{N\take M}\)$ is to see them as a nontrivial element in $\Ext^1\(N_{\PP(N)/\PP(N\oplus\cO_M)},\Omega_{\PP(N)}\)\cong\Ext^1\(\Omega_{N\take M/\PP(N)},h^0\(\EE_{\PP(N)}\)\)$ which is $\C$.} 
$\EE_{\PP(N)}^\vee\to N_{\PP(N)/\PP(N\oplus\cO_M)}[-1]$\footnote{The cocone of the cosection pretends to be a virtual rank zero dual perfect obstruction theory, but it is not really a dual perfect obstruction theory. Nevertheless, these are enough data to produce a degree zero virtual cycle.}. Following Kiem-Li \cite{KL}, these two define a degree zero reduced virtual cycle $[\PP(N)]^{\red}$, see Definition \ref{def:redcycle} in Section \ref{sect:Red}. 
\begin{Thm*}\label{main1}
The difference of the virtual cycles \eqref{virs} is 
$$
[N]^{\vir}_{\T} \ -\ [N]^{\vir}_{\sigma} \ = \ p_*[\PP\(N\)]^{\red} \ \in \ A_0\(M\),
$$ 
where $p: \PP\(N\) \to M$ is the projection morphism.
\end{Thm*}

\smallskip
We apply these ideas to the quintic and $t$-twisted quintic quasimap/GW invariants of Guo-Janda-Ruan \cite{GJR}\footnote{Felix Janda informed me that this is commonly called the $\cO_{\PP^4}(5)$-twisted invariants. It is different from the formal invariants of Lho-Pandharipande \cite{LP}.} 
Although Theorem \ref{main1} does not apply directly in this case, very similar techniques compute the difference between these invariants.
%

We denote by $\iota:Q^{\;\varepsilon}_g(X,d)\into Q^{\;\varepsilon}_g(\PP^4,d)$ the moduli spaces of degree $d$, $\varepsilon$-stable quasimaps with genus $g$, no marked points to a smooth quintic $3$-fold $X$ and $\PP^4$, $X\subset \PP^4$ \cite{CKM}. Then the quasimap invariant for $X$ is defined to be the degree of the virtual cycle
$$
[Q^{\;\varepsilon}_g(X,d)]^{\vir}\ \in\ A_0(Q^{\;\varepsilon}_g(X,d)).
$$
When $\varepsilon>2$, the spaces are moduli of stable maps, so the degree defines Gromov-Witten invariant. 

Consider the universal curve and map
\beq{universal}
\xymatrix@R=6mm{
C\ar[r]^-f\ar[d]_-\pi & \PP^4\\
Q^{\;\varepsilon}_g(\PP^4,d).
}
\eeq
In $g=0$, $Q^{\;\varepsilon}_0(\PP^4,d)$ is smooth, $R\pi_*f^*\cO(5)=\pi_*f^*\cO(5)$ is a vector bundle, and the section $\pi_*f^*s$ of $\pi_*f^*\cO(5)$ cuts out $Q^{\;\varepsilon}_0(X,d)\subset Q^{\;\varepsilon}_0(\PP^4,d)$, where $s\in \Gamma(\cO(5))$ is the defining section of $X=Z(s)\subset \PP^4$. It allows us to have the equivalence \cite{KKP} 
$$
\iota_*[Q^{\;\varepsilon}_0(X,d)]^{\vir}\= e(\pi_*f^*\cO(5)) \cap [Q^{\;\varepsilon}_0(\PP^4,d)].
$$
This motivates considering the cycles in $g>0$,
\beq{QLPclass}
\{e_{\T}\(R\pi_*f^*\cO(5)\ot \t^{-1}\)\cap [Q^{\;\varepsilon}_g(\PP^4,d)]^{\vir}\}_{t=0}\ \in\ A_0(Q^{\;\varepsilon}_g(\PP^4,d))
\eeq
even if $Q^{\;\varepsilon}_g(\PP^4,d)$ may not be smooth and $R^1\pi_*f^*\cO(5)$ need not vanish. 
Its degree is called the $t$-twisted invariant for $X$, so it is $t$-twisted Gromov-Witten invariant when $\varepsilon>2$. 

\begin{Thm*}\label{thm2}
The reduced cycle of the projectivisation of
$$
N\ :=\ \text{the moduli space of stable quasimap to $\PP^4$ with $p$-fields}
$$
gives the difference up to sign
\begin{align*}
\{e_{\T}\(R\pi_*f^*\cO(5)\ot\t^{-1}\) & \cap [Q^{\;\varepsilon}_g(\PP^4,d)]^{\vir}\}_{t=0} - \iota_*[Q^{\;\varepsilon}_g(X,d)]^{\vir} \\
&=(-1)^{5d+1-g}p_*[\PP(N)]^{\red}\ \in\ A_0\(Q^{\;\varepsilon}_g(\PP^4,d)\),
\end{align*}
where $p:\PP(N)\to Q^{\;\varepsilon}_g(\PP^4,d)$ is the projection morphism.
\end{Thm*}

\begin{Rmk}
We prove Theorems \ref{main1} and \ref{thm2} using the projective completion $\bN:=\PP(N\oplus \cO_M)$.
In \cite[Theorem 1.11=Theorem 3.21]{CJR}, Chen-Janda-Ruan proved a similar comparison result to Theorem \ref{thm2} using a different compactification of $N$. Recall the $p$-field space $N$ is the moduli space of stable objects of $(C,L,u,p)$, where
\begin{align*}
&\text{$C$ is a genus $g$ curve$\;$,}\ \text{ $L$ is a degree $d$ line bundle on $C$}, \\
&u\in \Gamma(C,L^{\oplus 5}),\ \ p\in\Gamma(C, L^{\otimes -5}\otimes\omega_C).
\end{align*} 
Their compactification $\bN_{\mathrm{CJR}}$ considers 
$$
\Gamma\(C,\PP(L^{\otimes -5}\otimes \omega_C\oplus\cO_C)\)\ \text{ instead of }\ \Gamma\(C, L^{\otimes -5}\otimes\omega_C\),
$$
together with some extra structures and conditions.
Then they constructed the canonical and reduced perfect obstruction theories of $\bN_{\mathrm{CJR}}$, where the latter gives the actual quintic theory and the former gives a different theory containing the $t$-twisted theory as a part of the torus localisation contributions. Their canonical perfect obstruction theory is equipped with the meromorphic cosection, inducing the regular cosection of the reduced perfect obstruction theory. It is surjective along the boundary, localising the reduced virtual cycle to the actual quintic cycle. However for us, we don't construct a reduced perfect obstruction theory on $\bN$, instead we use the meromorphic cosection of the canonical perfect obstruction theory to produce reduced virtual cycles in the sense of Kiem-Li \cite{KL}. 
\end{Rmk}

\smallskip
\subsection*{Acknowledgements} 
For more than a half year Richard Thomas had advised me the crucial idea, right direction and shaping the paper at our weekly meeting. I truly appreciate it. 

I should acknowledge an important contribution of Felix Janda and Yunfeng Jiang to this project. I am grateful to them for kindly sharing their note on projective completions and comments of finalising the paper. 

Honglu Fan suggested me that the quantum Lefschetz formula might be obtained by Theorem \ref{main1}. A regular online meeting with J\'er\'emy Gu\'er\'e and Mark Shoemaker was helpful to formulate Lemma \ref{HuHuHu}. I also thank Hyenho Lho, Evgeny Shinder, Bhamidi Sreedhar for useful comments.

In $2018$ after my advisor Bumsig Kim (1968-2021) -- who passed away during the preparation of this paper -- and I finished our paper \cite{KO}, we have talked if we can get an analogue of Theorem \ref{thm2}. He always encouraged me to grow up mathematically and personally. It was an honour and a privilege to be his student.

\subsection*{Notation}
For a morphism of schemes $f: X \to Y$ and a perfect complex $\cE$ on $Y$, we often denote by $\cE|_X$ the pullback $f^*\cE$. We use $\cE^*$ for the usual dual of $\cE$, whereas $\cE^\vee$ for the derived dual.

A vector bundle $E$ is sometimes thought of as its total space. For a morphism of vector bundles $f:E\to F$, $\ker f$ is sometimes thought of as the space $\Spec\(\Sym\(\coker f^*\)\)$.

\tableofcontents
\vspace{-1cm}

\section{Quasi-smooth derived schemes and stable quasimaps}
In this section, we review quasi-smooth derived schemes and moduli spaces of stable quasimaps. Though these two are not very close subjects, there are some common features to make the proofs of Theorem \ref{main1} and \ref{thm2} are equivalent. So after reviewing the two subjects, we summarise these in Section \ref{sect:SetUp} to establish the set-up for one proof.

\subsection*{Quasi-smooth derived scheme}
Let $M$ be a quasi-projective scheme with the perfect obstruction theory 
\beq{EM}
\EE_M=\{\,B^* \rt{d} A^*\,\}\ \To\ \LL_M,
\eeq
and $N := \Spec_{\cO_M} \Big( \Sym \(\text{coker} d^*\) \Big)$ be the abelian cone of the dual obstruction sheaf of $M$. 
By \cite[Lemma 2.1]{JT},
\beq{EMvee}
\EE_M^\vee|_N[1]=\{\,A|_N\rt{d^*}B|_N\,\}\ \To \ \LL_{N/M}
\eeq
is a perfect obstruction theory of $N$ relative to $M$. 
The relative perfect obstruction theory \eqref{EMvee} is indeed obtained by a cut-out model
\beq{localN}
\xymatrix@=18pt{
& A^*|_{B^*} \ar[d] \\  
 N\ =\ (d\circ\tau_{B^*})^{-1}(0)\ \subset\hspace{-8mm} & B^*,\ar@/^{-2ex}/[u]_{d\circ \tau_{B^*}}}
\eeq
where $\tau_{B^*}$ is the tautological section $\tau_{B^*} : \cO_{B^*} \to B^*|_{B^*}$.

Suppose moreover that $M$ is a cut-out of a smooth scheme $\cA$ by a section $s$ of $B$ so that \eqref{EM} is given by $d=(ds)^*|_M$. Here we regard $B$ as an extended bundle on $\cA$. Then $N$ is a critical locus $N=Z(df)$ of the pairing
$$
f\ :=\ \lan s|_{B^*}\, ,\, \tau_{B^*}\ran\ \in\ \Gamma\(\cO_{B^*}\).
$$
Hence we obtain a symmetric obstruction theory $\EE_N$ of $N$,
\beq{EEN}
\EE_N\ :=\ \{\, T_{B^*} \ \rt{d^2f}\ \Omega_{B^*}\,\}|_{N}. 
\eeq
Together with \eqref{EM}, \eqref{EMvee} it fits in the exact triangle $\EE_M|_N \to \EE_N \to \EE^\vee_M|_N[1]$. 

Though $M$ may not be obtained by a global cut-out model, $N$ is equipped with the canonical symmetric obstruction theory $\EE_N$ fit in the exact triangle if $M$ is (the classical truncation of) a quasi-smooth derived scheme -- because $N$ is (the classical truncation of) its $(-1)$-shifted cotangent bundle $T^\vee_M[-1]$. In this case local symmetric obstruction theories in \eqref{EEN} glue to $\EE_N$. In other words, the canonical symmetric lifting $\EE^\vee_M|_N \to \EE_M|_N$ of the morphism of cotangent complexes $\LL_{N/M}[-1] \to \LL_M|_N$ exists in the derived category, 
\beq{lift}
\xymatrix@R=6mm{
\EE^\vee_M|_N \ar[r] \ar[d] & \EE_M|_N \ar[d]\\
\LL_{N/M}[-1] \ar[r] & \LL_M|_N.
}
\eeq
$\EE_N$ is then its cone.

\smallskip

When $\T$ acts on $N$ fiberwise, $\EE^\vee_M|_N\ot \t^{-1}[1] \to \LL_{N/M}$ is the $\T$ relative perfect obstruction theory \eqref{EMvee}. The lifting \eqref{lift} becomes $\T$ invariant one $\EE^\vee_M|_N\ot \t^{-1} \to \EE_M|_N$. 
Then the cone defines the $\T$ canonical symmetric obstruction theory, which is locally
$$
\EE_N\ =\ \{\, T_{B^*}\ot \t^{-1} \ \To\ \Omega_{B^*}\,\}|_{N}.
$$
On the fixed locus $M$, $\EE_N|_M$ is decomposed into $\EE_M\oplus \EE^\vee_M[1]\ot\t^{-1}$. Hence $[N]^{\vir}_{\T}$ in \eqref{virs} is $\T$ localisation of $[N]^{\vir}$.

\smallskip
Importantly with the $\T$ action, the tautological section $\tau_{B^*}$ becomes Euler vector field $\tau_{B^*}:\cO_{B^*}\to B^*|_{B^*}\ot \t.$ Then by the (equivariant version of) cut-out model \eqref{localN} of $N$, we see its dual defines a cosection $\EE^\vee_M|_N\to\cO_N[-1]\ot\t.$ Hence the composition with $\EE^\vee_N\to \EE^\vee_M|_N$ defines an equivariant cosection
$$
\sigma\ :\ \EE^\vee_N\ \To\ \cO_N[-1]\ot\t.
$$
Note that $h^1(\sigma)$ is surjective where $\tau_{B^*}|_N$ is nonzero, which is complement of $M$, $N\take M$. The localisation of $[N]^{\vir}$ by $\sigma$ is then $[N]^{\vir}_{\sigma}$ in \eqref{virs}.

\medskip
\subsection*{Stable quasimaps}
Let $M:=Q^{\;\varepsilon}_g(\PP^4,d)$ be the moduli space of degree $d$, $\varepsilon$-stable quasimaps with genus $g$, no marked points to $\PP^4$ \cite{CKM}. The complex $\(R\pi_*f^*(\cO(1)^{\oplus 5})\)^\vee$ defined by using $\pi$ and $f$ in the universal family \eqref{universal} gives a relative perfect obstruction theory $\EE_{M/S}$ of $M$ over the moduli space of prestable curves of genus $g$, with line bundles of degree $d$, which we denote by $S$. Although we may not know if $M$ is projective, by \cite[Step 1 in Section 3.2.1]{KO}, $\EE_{M/S}$ has enough $2$-term, locally free representatives, i.e. for any $2$-term representative of coherent sheaves of $\EE_{M/S}$, there exists a $2$-term, locally free representative and a chain map from it to given representative. We pick one such
$$
\EE_{M/S}\=\{\;B'^*\ \To\ A'^*\;\}. 
$$

The $p$-field space $N$ is defined to be $N:=\Spec \Sym h^{1}\(R\pi_*f^*\cO(5)\)$ \cite{CL12, FJR, CFGKS}. Again by \cite[Lemma 2.1]{JT}, $R\pi_*f^*\cO(5)|_N[1]$ becomes a relative perfect obstruction theory $\EE_{N/M}$. By \cite[Step 1 in Section 3.2.1]{KO}, $R\pi_*f^*\cO(5)[1]$ has enough $2$-term, locally free representatives. We pick one $A\to B$, giving
$$
\EE_{N/M}\=\{\;A|_N\ \rt{d^*}\ B|_N\;\}.
$$
With this notation, $N$ has the same (relative) cut-out model \eqref{localN}.

Note that by its construction in \cite{CL12, FJR, CFGKS}, $(N,\EE_{N/M})$ is actually a pullback space by $M\to S$ so that $\EE_{M/S}|_N\oplus\EE_{N/M}$ defines a relative perfect obstruction theory $\EE_{N/S}$. Then the cone defines a virtual rank zero perfect obstruction theory of $N$,
$$
\EE_N\ :=\ \Cone\(\;\EE_{M/S}|_N\oplus\EE_{N/M}\ \To\ \LL_S|_N\).
$$
Here the restriction of the cotangent complex $\LL_S|_N$ is a bundle on $N$ since $S$ is a smooth Artin stack \cite{CKM} and $N$ is a Deligne-Mumford stack. 

\smallskip
Considering the fiberwise $\T$ action, the complex $\EE_{N/M}$ becomes an equivariant complex $R\pi_*f^*\cO(5)|_N[1]\ot\t^{-1}$ with nonzero weights, and hence $\T$-localised cycle is
\begin{align}\label{TloCalN}
[N]_{\T}^{\vir}&\=\left\{\frac{[M]^{\vir}}{e_{\T}\((R\pi_*f^*\cO(5))^\vee[-1]\ot\t\)}\right\}_{t=0}\nonumber\\
&\=(-1)^{5d+1-g}\{e_{\T}\(R\pi_*f^*\cO(5)\ot\t^{-1}\)\cap [M]^{\vir}\}_{t=0}\ \in\ A_0\(M\).
\end{align}

\smallskip
Chang-Li \cite{CL12} firstly introduced a cosection of $N$ when $\varepsilon>2$ whose degeneracy locus is $Q^{\;\varepsilon}_g(X,d)$, showing that the cosection localised invariant of $N$ is equivalent to GW invariant of the quintic $3$-fold $X$ up to sign. By several works \cite{KO, CL20, CJW, Pi} after this, it has been studied the cosection localised cycle of $[N]^{\vir}$ is equal to $(-1)^{5d+1-g}[Q^{\;\varepsilon}_g(X,d)]^{\vir}$ for any $\varepsilon$. This cosection is a sum of two pieces defined on each direct summand of $\EE^\vee_{N/S}$. Here, we are interested in the piece on $\EE^\vee_{M/S}|_N$. The below is a construction.

In \cite[Step 2 in Section 3.2.1]{KO}, a chain map representative between $\EE^\vee_{M/S}[1]=R\pi_*f^*(\cO(1)^{\oplus 5})$ and $\EE_{N/M}|_M=R\pi_*f^*\cO(5)$ induced by the defining equation of $X$ is constructed,
\beq{ChainABAB}
\xymatrix@R=6mm@C=6mm{
\EE^\vee_{M/S}[1]=R\pi_*f^*(\cO(1)^{\oplus 5}) \ar[d]  &\hspace{-15mm} = \hspace{-20mm}& \{ \hspace{-12mm}& A'\ar[r]\ar[d] &B'\ar[d] & \hspace{-13mm}\} \\
R\pi_*f^*\cO(5) &\hspace{-15mm} = \hspace{-20mm}& \{ \hspace{-12mm}& A\ar[r] & B& \hspace{-12mm}\}
}
\eeq
by choosing suitable representatives $A,B,A',B'$. Composing with $B'|_{B^*}\to B|_{B^*}$ defined above in \eqref{ChainABAB}, the dual tautological section $\tau_{B^*}:\cO_{B^*}\to B^*|_{B^*}\ot\t$ defines a homomorphism $B'|_{B^*}\to \cO_{B^*}\ot\t$. From the cut-out model \eqref{localN} of $N$, it defines a cosection $\EE^\vee_{M/S}|_N\to \cO_{N}[-1]\ot\t$ on $N$. Since the composition $\LL^\vee_S|_N[-1]\to \EE^\vee_{M/S}|_N\to \cO_{N}[-1]\ot\t$ is zero by \cite[Equation (3.16)]{KO}, we obtain a cosection
$$
\sigma\ :\ \EE^\vee_N\ \To\ \cO_N[-1]\ot\t.
$$ 
Then its degeneracy locus is $M$ \cite[Equation (3.8)]{KO}, and hence by \cite{KO, CL20, CJW, Pi} $\sigma$-localised cycle is
\beq{CoSecN}
[N]_{\sigma}^{\vir}\=(-1)^{5d+1-g}\iota_*[Q^{\;\varepsilon}_g(X,d)]^{\vir}\ \in\ A_0\(M\).
\eeq
By \eqref{TloCalN}, \eqref{CoSecN}, the statements of Theorem \ref{main1} and \ref{thm2} became equivalent now.

\medskip
\subsection{Set-up}\label{sect:SetUp}
Let $M$ be a finite type, separated Deligne-Mumford stack over a smooth Artin stack $S$ with the relative perfect obstruction theory $\EE_{M/S}$. Suppose it has enough $2$-term, locally free representatives so that we can pick one such
$$
\EE_{M/S}\ :=\ \{B'^*\ \To\ A'^*\}.
$$

For a morphism of vector bundles $d:B^*\to A^*$ on $M$, we define $N:=\ker d$. Then by \cite[Lemma 2.1]{JT}, the relative perfect obstruction theory is
\beq{setupEEN}
\EE_{N/M}\ =\ \{A|_N\ \rt{d^*}\ B|_N\}.
\eeq
So assume that $\EE_{N/M}$ has enough $2$-term, locally free representatives.

An important assumption is that there exists a lift $\EE_{N/M}[-1]\to \EE_{M/S}|_N$ of $\LL_{N/M}[-1]\to \LL_{M/S}|_N$ so that its cone defines a perfect obstruction theory $\EE_{N/S}$. Moreover we assume that $\dim S+\rk\EE_{N/S}=0$ so that the perfect obstruction theory of $N$
$$
\EE_N\ :=\ \Cone\(\EE_{N/S}[-1]\ \To\ \LL_S|_N\)
$$
defines a degree zero virtual cycle $[N]^{\vir}\in A_0(N)$.

\smallskip
The fiberwise $\T$ action on $N$ localises the virtual cycle $[N]^{\vir}$ via torus localisation \cite{GP},
$$
[N]^{\vir}_{\T}\ :=\ \left\{\frac{[M]^{\vir}}{e_{\T}\(\EE^\vee_{N/M}|_M\)}\right\}_{t=0}\ \in\ A_0\(M\).
$$

\smallskip
Assume that there exists a chain map between $\EE^\vee_{M/S}[1]$ and $\EE_{N/M}|_M$,
\beq{ChainAB}
\xymatrix@R=6mm@C=6mm{
\EE^\vee_{M/S}[1] \ar[d]  &\hspace{-15mm} = \hspace{-20mm}& \{ \hspace{-12mm}& A'\ar[r]\ar[d] &B'\ar[d] & \hspace{-13mm}\} \\
(\EE_{N/M}|_M)|_{\t=1} &\hspace{-15mm} = \hspace{-20mm}& \{ \hspace{-12mm}& A\ar[r] & B& \hspace{-12mm}\}
}
\eeq
for suitable choices of $A,B,A',B'$. With $B'\to B$ above in \eqref{ChainAB}, the tautological section $\tau_{B^*}:\cO_{B^*}\to B^*|_{B^*}\ot\t$ defines a cosection $\EE^\vee_{M/S}|_N\to \cO_N[-1]\ot\t$ on $N$, and hence it defines a cosection $\EE^\vee_{N/S}\to \cO_N[-1]\ot\t$. We assume further that the composition $\LL^\vee_S|_N[-1] \to \EE^\vee_{N/S}\to \cO_N[-1]\ot\t$ is zero so that the cosection
$$
\sigma\ :\ \EE^\vee_N\ \To\ \cO_N[-1]\ot\t 
$$
is defined. Assuming the degeneracy locus is contained in $M$, we get $\sigma$-localised cycle $[N]^{\vir}_{\sigma}\in A_0(M)$ of $[N]^{\vir}$.

\smallskip
Note that we take $S=\Spec \C$ and $\eqref{ChainAB}=\id$ for a quasi-smooth derived scheme $M$ to be in the set-up. In the following Sections, we assume $N\neq M$ since Theorems \ref{main1} and \ref{thm2} are obvious in this case.

\section{Virtual cycle on the projective completion}
We start with the set-up in Section \ref{sect:SetUp}. On the projective completion of $N$ with the $\T$ action
$$
\overline{N}\ :=\ \PP\(N \ot \t \oplus \cO_M\),
$$
we would like to define a $\T$ perfect obstruction theory, extending $\EE_N$. Then by virtual localisation \cite{GP} we localise the virtual cycle to the fixed locus
$$
\bN^\T \ \cong\ M\ \cup\ D
$$
where $M \into \bN$ is the zero section and $D := \PP\(N\)\into \bN$ is the infinity divisor. We prove its contribution lying on $M$ is $[N]^{\vir}_{\T}$ and that on $D$ is zero so that the pushdown of the virtual cycle to $M$ is $[N]^{\vir}_{\T}$.

\subsection{Extended perfect obstruction theory}\label{sect:LPOT}
To extend the perfect obstruction theory $\EE_{N/S}$ to $\bN$, we consider the quotient expression of $\bN$ -- it is a (GIT) quotient of $N \times \C$ by $\C^*$. Then we use the perfect obstruction theory of $N\times\C$
\beq{obstimesC}
\Phi \ :\ \EE_{N \times \C/S} \ :=\ \EE_{N/S} \boxplus \cO_{\C} \ \To\ \LL_{N \times \C/S} \ \cong\ \LL_{N/S} \boxplus \cO_{\C}.
\eeq

Let $\T':=\C^*$ be a torus, different from $\T$. We denote by $\t'$ the standard weight $1$ representation of $\T'$. Then we consider a $\T \times \T'$ action on $N \times \C$ to be $\(N \ot \t \ot \t'\) \oplus \(\cO_M \ot \t'\).$ Then $\bN$ is the quotient of $N \times \C$ by $\T'$ after removing the zero section
$$
\bN \ \cong\ \wt{N}/\,\T', \ \ \wt{N}:=\(N\times\C\)\take M.
$$
We denote by $q: \wt{N} \to \bN$ the quotient morphism. Sometimes thinking of $\wt{N}$ as the punctured tautological bundle $\wt{N}\cong\cO_{\bN}(-D)\take \bN$ will be useful.

By abuse of notation, we regard the perfect obstruction theory $\Phi$ \eqref{obstimesC} on $N \times \C$ as a $\T \times \T'$ equivariant one. Then the restriction 
$$
\EE_{\wt{N}/S}\ :=\ \EE_{N \times \C/S}|_{\wt{N}}
$$ 
defines 
a perfect obstruction theory since $\wt{N} \into N\times\C$ is open. 

As how we get $\bN$ as a quotient of $\wt{N}$, we would like to define {\em the quotient of $\EE_{\wt{N}/S}$}, providing a perfect obstruction theory of $\bN$ over $S$. To see where it sits on, we consider the exact triangle of cotangent complexes
$$
q^*\LL_{\bN/S}\ \To\ \LL_{\wt{N}/S} \ \To\ \Omega_q
$$
induced by a smooth morphism $q$. 
We observe from this that the place where $\EE$ is in the diagram below \eqref{twotriangle} of triangles
\beq{twotriangle}
\xymatrix@R=5mm@C=13mm{
\EE\ar[d]\ar[r]&\EE_{\wt{N}/S} \ar[d]_-{\Phi} \ar[r] &\Omega_q \ar@{=}[d]\\
q^*\LL_{\bN/S}\ar[r]&\LL_{\wt{N}/S}  \ar[r] & \Omega_q.
}
\eeq
should be for the pullback perfect obstruction theory of $\bN$ by $q$. Using the equivalence of the bounded derived categories $q^*:D\(\bN\)\to D_{\T'}\(\wt{N}\)$ \cite[Proposition 2.2.5]{BL}, we define the quotient.

\begin{Def} \label{def:twotriangle}
We define {\em the quotient of $\EE_{\wt{N}/S}$} to be 
$$
\EE_{\bN/S}\ :=\ (q^*)^{-1}\EE\=(q^*)^{-1}\Cone\(\EE_{\wt{N}/S}\to\Omega_q\)[-1].
$$
\end{Def}

Again using the equivalence $q^*:D\(\bN\)\to D_{\T'}\(\wt{N}\)$, we obtain a morphism
\beq{EElog}
\EE_{\bN/S}\ \To\ \LL_{\bN/S} \ \in\ D\(\bN\)
\eeq
whose pullback by $q^*$ recovers $\EE\to q^*\LL_{\bN/S}$.

\begin{Prop} \label{prop:EEbN}
\eqref{EElog} is a perfect obstruction theory of $\bN$ over $S$, extending $\EE_{N/S}$, i.e. $\EE_{\bN/S}|_N\cong\EE_{N/S}$.
\end{Prop}
\begin{proof}
Let us prove $\EE$ is a perfect complex of amplitude $[-1,0]$ first. Since $h^i\(\EE_{\wt{N}/S}\)=0$ for $i\neq -1,0$, we have $h^i\(\EE\)=0$ for $i\neq -1,0,1$. The induced morphism $h^0(\LL_{\wt{N}/S})\to\Omega_q$ is onto, 
showing $h^1(\EE)=0$. Hence $\EE$ is a perfect complex of amplitude $[-1,0]$. Next, by comparing the long exact sequences induced by each row of \eqref{twotriangle}, we can check that $h^0(\EE)\to h^0(q^*\LL_{\bN/S})$ is an isomorphism, and $h^{-1}(\EE)\to h^{-1}(q^*\LL_{\bN/S})$ is onto. 

Since $q$ is smooth, $\EE_{\bN/S}$ is also a perfect complex of amplitude $[-1,0]$, $h^0(\EE_{\bN/S})\to h^0(\LL_{\bN/S})$ is an isomorphism, and $h^{-1}(\EE_{\bN/S})\to h^{-1}(\LL_{\bN/S})$ is onto. Hence \eqref{EElog} is a perfect obstruction theory of $\bN$ over $S$. 

Since the morphism $q$ on $q^{-1}(N)=N\times\C^*\subset \wt{N}$ is the projection to $N$, we have $\EE|_{q^{-1}(N)}\cong\EE_{N/S}|_{q^{-1}(N)}$ by the construction of $\EE$ in \eqref{twotriangle}. So 
\begin{align*}
\EE_{\bN/S}|_N\,\cong\,\((q^*)^{-1}\EE\)|_{N}\,&\cong\,(q^*)^{-1}\(\EE|_{q^{-1}(N)}\)\\
\,& \cong\,(q^*)^{-1}\(\EE_{N/S}|_{q^{-1}(N)}\)\,\cong\,\EE_{N/S},
\end{align*}
proving $\EE_{\bN/S}$ is an extension of $\EE_{N/S}$.
\end{proof}

\subsection{Relative perfect obstruction theory}
As how we define $\EE_{\bN/S}$, a relative perfect obstruction theory $\EE_{\bN/M}$ of $\bN$ over $M$ can be defined as the quotient of $(\EE_{N/M}\boxplus\cO_{\C})|_{\wt{N}}$ in Definition \ref{def:twotriangle}. This is an extension of the relative perfect obstruction theory \eqref{setupEEN} with the $\T$ action
\beq{TsetupEEN}
\EE_{N/M}\,=\,\{A|_N\ot\t^{-1} \rt{d^*} B|_N\ot\t^{-1}\}.
\eeq 
In this section we provide a cut-out expression of $\bN$ defining $\EE_{\bN/M}$. 

Let $\bcB$ denote the projective completion $\PP\((B^*\ot \t) \oplus \cO_M\)$, which is smooth over $M$. Then $\bN$ is a cut-out of $\bcB$ as in the following extended picture of \eqref{localN},
\beq{locbN}
\xymatrix@=18pt{
& A^*|_{\bcB}(\ccD) \ot \t  \ar[d] \\  
 \bN\ =\ (d\circ\overline{\tau})^{-1}(0)\ \subset\hspace{-12mm} & \bcB,\ar@/^{-2ex}/[u]_{d\circ \overline{\tau}}}
\eeq
where $\ccD:= \PP\(B^*\)\into \bcB$ is the infinity divisor. The section $\overline{\tau}$ is given by the tautological line bundle 
\beq{EulerbcBM}
0\, \to\, \cO_{\bcB}(-\ccD) \ \rt{\overline{\tau}\, \oplus\, s_{\ccD} }\ \(B^*|_{\bcB} \ot \t\) \oplus \cO_{\bcB}\, \to\, T_{\bcB/M}(-\ccD) \, \to\, 0.
\eeq

\begin{Lemma}\label{Lem:Er}
The relative perfect obstruction theory $\EE_{\bN/M}$ is represented by
\beq{relEElog}
\EE_{\bN/M} \ \cong\ \{\,A|_{\bcB}(-\ccD) \ot \t^{-1} \ \rt{d(d\circ \overline{\tau})^*}\ \Omega_{\bcB/M} \,\}|_{\bN}.
\eeq
\end{Lemma}
\begin{proof}
We use the identifications of the tangent bundle of $q$
\beq{Tq}
T_q\ \cong\ \cO_{\wt{N}}\ 
\cong\ \cO_{\bN}(-D)|_{\wt{N}}\ot \t',
\eeq
obtained by the Euler sequence of $\PP\(\cO_{\bN}(-D)\oplus\cO_{\bN}\)$. Here, we consider $\wt{N}$ as the punctured tautological bundle $\cO_{\bN}(-D)\take \bN$.

By Definition \ref{def:twotriangle}, the pullback $q^*\EE_{\bN/M}$ is the cocone of 
\beq{this}
\EE_{N/M}\ot\t'^{-1}\boxplus\cO_{\C}\ot\t'^{-1} \ \To\ \Omega_q.
\eeq
Using the global resolution of $\EE_{N/M}$ \eqref{TsetupEEN} and \eqref{Tq}, we observe that the above morphism \eqref{this} is represented by a chain map 
\beq{chainAB}
\xymatrix@R=5mm{
B|_{\wt{\cB}}\ot(\t\ot\t')^{-1} \oplus \cO_{\wt{\cB}}\ot\t'^{-1} \ar[r] & \cO_{\wt{\cB}} \\ 
A|_{\wt{\cB}}\ot(\t\ot\t')^{-1}. \ar[u]^-{d^*\oplus\,0} 
}
\eeq
on $\wt{N}\subset \wt{\cB}:=(B^*\times\C)\take M$. Note that the composition 
$$
A|_{\wt{\cB}}\ot(\t\ot\t')^{-1}\ \rt{d^*}\ B|_{\wt{\cB}}\ot(\t\ot\t')^{-1}\ot(\t\ot\t')^{-1}\ \To\ \cO_{\wt{\cB}}
$$ 
is zero on $\wt{N}$. The horizontal morphism of \eqref{chainAB} is a part of the dual Euler sequence $\(\eqref{EulerbcBM}\otimes\, \cO_{\bcB}(\ccD)\)|_{\wt{\cB}}$. So its kernel is $\Omega_{\bcB/M}|_{\wt{\cB}}$. The bundle at the bottom of \eqref{chainAB} is $\(A|_{\bcB}(-\ccD)\)|_{\wt{\cB}}\ot\t^{-1}$ by \eqref{Tq}, which induces $\cO_{\bcB}(\ccD)|_{\wt{\cB}}=\t'$. So $(q^*)^{-1}\(\eqref{chainAB}|_{\wt{N}}\)$ gives the representative \eqref{relEElog}.
\end{proof}

\subsection{Virtual cycle}
Now we define a perfect obstruction theory of $\bN$ to be
$$
\EE_{\bN}\ :=\ \Cone\(\EE_{\bN/S}[-1]\ \To\ \LL_S|_{\bN}\),
$$
which produces a degree zero virtual cycle
$$
[\bN]^{\vir} \ \in\ A_0\(\bN\).
$$
By Proposition \ref{prop:EEbN}, $\EE_{\bN}$ is an extension of $\EE_N$. 

\begin{Prop} \label{Prop:log}
The virtual cycle $[\bN]^{\vir}$ pushes down to $[N]^{\vir}_\T$ defined in \eqref{virs} by the projection morphism $\pi: \bN \to M$,
$$
[N]^{\vir}_\T \ =\ \pi_*[\bN]^{\vir} \ \in\ A_0\(M\).
$$
\end{Prop}
\begin{proof}
We prove it via $\T$ virtual localisation \cite{GP}. To apply it to $[\bN]^{\vir}$, we need to investigate the fixed and moving parts of the perfect obstruction theory $\EE_{\bN}$ on the fixed locus $\bN^\T=M\cup D$. To see this, we use a triangle 
\beq{tritritri}
\EE_{M/S}|_{\bN}\ \To\ \EE_{\bN/S}\ \To\ \EE_{\bN/M},
\eeq
induced by a diagram of triangles we have constructed so far
$$
\xymatrix@R=5mm@C=7mm{
q^*\EE_{M/S}|_{\bN}\ar@{=}[r]\ar@{-->}[d]& \EE_{M/S}|_{\wt{N}}\ar[d]\\
q^*\EE_{\bN/S}\ar[r]\ar@{-->}[d]&\EE_{\wt{N}/S} \ar[d] \ar[r] &\Omega_q \ar@{=}[d]\\
q^*\EE_{\bN/M}\ar[r]&\EE_{N/M}\boxplus\cO_{\C}|_{\wt{N}}  \ar[r] & \Omega_q.
}
$$
Note that the middle and bottom horizontals are coming from the definitions of $\EE_{\bN/S}$ and $\EE_{\bN/M}$ Definition \ref{def:twotriangle}. The mid-vertical is obtained by the triangle
$$
\EE_{M/S}|_N\ \To\ \EE_{N/S}\ \To\ \EE_{N/M}.
$$
Then we get the left vertical, giving \eqref{tritritri}. It tells us that the moving part of $\EE_{\bN}|_{\bN^{\T}}$ is isomorphic to that of $\EE_{\bN/M}|_{\bN^{\T}}$.

On $M$, we have $\EE_{\bN/M}|_M\cong \EE_{N/M}|_M$ which is the moving part. So the contribution of $\T$ localisation of $[\bN]^{\vir}$ on $M$ is $[N]^{\vir}_{\T}$.

It remains to show that the contribution on $D$ is zero. We use the representative of $\EE_{\bN/M}|_D$ obtained by Lemma \ref{Lem:Er}
$$
\EE_{\bN/M}|_D \ \cong\ \{\,A|_{D}(-\ccD) \ot \t^{-1} \ \rt{d(d\circ \overline{\tau})^*}\ \Omega_{\bcB/M}|_D \,\}|_{\bN}.
$$
The bundle $A|_{D}(-\ccD) \ot \t^{-1}$ is fixed since the tautological line bundle $\cO_{D}(-\ccD)$ is contained in $B^*|_{D}\ot\t$. Whereas $\Omega_{\bcB/M}|_D$ contains a conormal bundle $N^*_{D/\bN}$ which is not fixed. Hence the moving part of $\EE_{\bN/M}|_D$ has virtual rank $1$, which implies the fixed part of $\EE_{\bN}|_D$ has virtual rank $-1$. So the contribution on $D$ is zero.
\end{proof}

\section{Reduced cycle on the projectivisation}
In this section we find an extension
$$
\bsigma\ :\ \EE_{\bN}^\vee \ \To\ \cO_{\bN}(D)[-1] \ot \t,
$$
of the cosection $\sigma: \EE_N^\vee\to \cO_N[-1]\ot\t$. Assuming $h^1(\bsigma)$ is surjective on $D\subset \bN$ it defines the reduced cycle $[\PP(N)]^{\red}$. Then we prove $\bsigma$-localised cycle of $[\bN]^{\vir}$ is a sum $[N]^{\vir}_{\sigma}+[\PP(N)]^{\red}\in A_0(M\cup D)$.

\subsection{Extended twisted cosection}
The Euler sequence on $\bcB$ induces a homomorphism
$$
B|_{\bcB}\ \rt{\btau}\ \cO_{\bcB}(\ccD)\ot\t,
$$
which is surjective on $\ccD$, extending the dual tautological section $\tau_{B^*}$. Using the cut-out model of $\bN\subset \bcB$ \eqref{locbN} we see the composition with $B'|_{\bcB}\to B|_{\bcB}$ in \eqref{ChainAB} defines a cosection on $\bN$ 
$$
\EE^\vee_{M/S}|_{\bN}\ \To\ \cO_{\bN}(D)[-1]\ot \t.
$$
The composition $\LL_S^\vee[-1]|_{\bN}\to \EE^\vee_{\bN/S}\to \EE^\vee_{M/S}|_{\bN}\to \cO_{\bN}(D)[-1]\ot \t$ is zero because it is zero on $N$ by the assumption in Section \ref{sect:SetUp}. Hence it defines a cosection
$$
\bsigma\ :\ \EE^\vee_{\bN}\ \To\ \cO_{\bN}(D)[-1]\ot \t.
$$
Obviously it is an extension of $\sigma: \EE^\vee_N\to \cO_N[-1]\ot\t$.

\subsection{Reduced cycle}\label{sect:Red}
$\PP(N)=D$ is equipped with the perfect obstruction theory $\EE_{D}:=\EE_{\bN}|_D^{\mathrm{fix}}$ (of virtual rank $-1$) and the morphism $\bsigma_D:=\bsigma|_{\EE^\vee_D}:\EE^\vee_D\to \cO_D(D)[-1]\ot\t$. Note that $\EE_D$ is described as a quotient
$$
\EE_{D}\ \cong\ \Cone\(\,\EE_N|_{N\take M}\ \To\ \Omega_{(N\take M)/D}\)[-1]
$$
through the equivalence $D_{\C^*}\(N\take M\)\cong D\(D\)$. Now we assume the surjectivity of $h^1(\bsigma)$ on $D$, inducing the surjectivity of $h^1(\bsigma_D)$. In this case, Kiem-Li proved the cone reduction property \cite[Corollary 4.5]{KL}
\beq{INCLUSION}
\fC_D\ \subset\ h^1/h^0\(\Cone(\bsigma_D)\),
\eeq
where $\fC_D$ is the intrinsic normal cone of $D$ and $h^1/h^0$ denotes the cone stack associated to a complex defined in \cite{BF}. The surjectivity of $h^1(\bsigma_D)$ allows $h^1/h^0(\Cone(\bsigma_D))$ to be actually a bundle stack so that the Gysin map
$$
0^!_{h^1/h^0(\Cone(\bsigma_D))}\ :\ A_{0}\Big(h^1/h^0\(\Cone(\bsigma_D)\)\Big)\ \To\ A_{0}\(D\)
$$
is defined \cite{Kr}. 
\begin{Def}\label{def:redcycle}
{\em The reduced cycle} is defined to be 
$$
[\PP(N)]^{\red}\ :=\ 0^!_{h^1/h^0(\Cone(\bsigma_D))}[\fC_D] \ \in\ A_{0}\(D\).
$$
\end{Def}

Note that the inclusion \eqref{INCLUSION} may not be a scheme theoretic embedding, but a set-theoretic embedding. Hence $\Cone(\bsigma_D)^\vee$ may not be a perfect obstruction theory in general.

\begin{Thm} \label{main2}
We obtain the following comparison result
$$
[N]^{\vir}_{\T}\ -\ [N]^{\vir}_{\sigma}\=p_*[\PP(N)]^{\red}\ \in\ A_0(M),
$$
where $p:\PP(N)\to M$ is the projection morphism.
\end{Thm}
\begin{proof}
Proposition \ref{Prop:log} implies
$$
[N]^{\vir}_{\T}\=\pi_*[\bN]^{\vir}\ \in\ A_0(M).
$$
So it is enough to show that $[N]^{\vir}_{\sigma}+[\PP(N)]^{\red}$ is a localisation of $[\bN]^{\vir}$.

Take any global representative 
$$
\EE_{\bN}\=\{\;E^*\ \To\ T^*\}.
$$ 
Then the Behrend-Fantechi cone $C_{\bN}$ (of $\dim=\rk T=\rk E$) lies in $E$ and the virtual cycle is its intersection with the zero section
\beq{FulDef}
[\bN]^{\vir}\=0^!_E[C_{\bN}].
\eeq
The twisted cosection $\bsigma$ induces a homomorphism
$$
\bsigma\ :\ E\ \To\ \cO_{\bN}(D),
$$
which we denote also by $\bsigma$ by abuse of notation. It is onto outside of $M\into \bN$. Hence on the blowup $b:Bl:=Bl_M\bN\to \bN$ with the exceptional divisor $D_0$, $\bsigma$ induces a surjection
$$
\bsigma_{Bl}\ :\E|_{Bl}\ \twoheadrightarrow\ \cO_{Bl}\(D-D_0\).
$$
Let $F:=\ker\(\bsigma_{Bl}\)$. Then the disjoint union $E|_M \cup F$ surjects to the kernel of $\bsigma$,
$$
E|_M\, \cup\, F\ \twoheadrightarrow\ \ker\(\bsigma\)\ \subset E.
$$ 
Since this induces a surjection of Chow groups, the cone reduction property $C_{\bN}\subset\ker\(\bsigma\)$ for the cosection $\bsigma$ \cite[Corollary 4.5]{KL} tells us that there exist cycles 
\begin{align}\label{CN12}
&[C_1]\ \in\ A_{\rk E}\(E|_M\), \ \ [C_2]\ \in\ A_{\rk E}\(F\) \text{ such that} \nonumber\\
&[C_{\bN}]\=[C_1]+[C_2]\,\in\,A_{\rk E}\(E\) \text{ after pushforwards}.
\end{align}
Hence we obtain a localisation
\begin{align*}
&0^!_{E|_M}[C_1] + b_*\((D-D_0)\cap0^!_F[C_2]\)\ \in\ A_{0}\(M\cup D\)\\ \nonumber
&\ \ \Mapsto 0^!_{E}[C_1] + b_*\(0^!_{\cO(D-D_0)}0^!_{F}[C_2]\)\\ \nonumber
&\ \ \ \ \ \ \, = 0^!_{E}[C_1] + b_*\(0^!_{E|_{Bl}}[C_2]\) \stackrel{\eqref{CN12}}{=} 0^!_E[C_{\bN}]\stackrel{\eqref{FulDef}}{=} [\bN]^{\vir}\ \in\ A_0\(\bN\).
\end{align*}
Note that by definition of cosection localisation \cite[Section 2]{KL}, we have
$$
0^!_{E|_M}[C_1]\; -\; b_*\(D_0\;\cap\;0^!_F[C_2]\)\, =\,[N]^{\vir}_{\sigma}.
$$
So it remains to show that
\beq{final1}
b_*\(D\,\cap\,0^!_F[C_2]\)\=[\PP(N)]^{\red}.
\eeq

Since $D$ does not meet $D_0$, we can restrict $F$ and $C_2$ to $Bl\take D_0 \cong N_{D/\bN}$, having
\begin{align}\label{final2}
b_*\(D\,\cap\,0^!_F[C_2]\)&\=D\,\cap\,0^!_{F|_{N_{D/\bN}}}[C_2|_{N_{D/\bN}}]\\
&\=0^!_{F|_D}\(F|_D\,\cap\,[C_2|_{N_{D/\bN}}]\).\nonumber
\end{align}
The restriction $[C_2|_{N_{D/\bN}}]$ is the cycle representing Behrend-Fantechi cone of $N_{D/\bN}$
$$
C_{N_{D/\bN}}\ \subset\ F|_{N_{D/\bN}} \ \subset\ E|_{N_{D/\bN}},
$$
obtained by the representative $\{E|^*_{N_{D/\bN}}\to T|^*_{N_{D/\bN}}\}=\EE_{\bN}|_{N_{D/\bN}}$. 

We claim that the intersection with the infinity divisor $F|_D\subset F|_{N_{D/\bN}}$ is the Behrend-Fantechi cone of $D$, 
\beq{CLAIM}
F|_D\cap C_{N_{D/\bN}} \=C_D\ \subset\ F|_D\ \subset\ E|_D
\eeq
obtained by the representative $\{E^*|_D\to\frac{T^*|_D}{N^*_{D/\bN}}\}=\EE_{\bN}|_D^{\mathrm{fix}}$, as cycles. It is enough to show that $C_D=C_{N_{D/\bN}}|_D$ as Deligne-Mumford stacks. But this is more or less obvious since $N_{D/\bN}$ is a bundle on $D$. Picking any local smooth embedding $D\subset \cU$, we obtain a Cartesian diagram
$$
\xymatrix@R=6mm{
D\,\ar@{^(->}[r]\ar@{^(->}[d]& \cU\ar@{^(->}[d] \\
N_{D/\bN}\ar@{^(->}[r] & \cU\times\C,
}
$$ 
by assuming $N_{D/\bN}$ is a trivial bundle on $D$ locally. Then we have 
\begin{align*}
&C_{D/\cU}= C_{N_{D/\bN}/\cU\times\C}|_D\\
&\Longrightarrow\ \left[\frac{C_{D/\cU}}{T_{\cU}|_D}\right]\times_{\left[\frac{E|_D}{T|_D/N_{D/\bN}}\right]}E|_D=\left[\frac{C_{N_{D/\bN}/\cU\times\C}|_D}{T_{\cU\times\C}|_D}\right]\times_{\left[\frac{E|_D}{T|_D}\right]}E|_D\\
&\Longrightarrow\ C_D=C_{N_{D/\bN}}|_D
\end{align*}
since the last equality is a gluing of the middle equality. So the claim \eqref{CLAIM} is true. By \eqref{final2}, \eqref{CLAIM}, we have \eqref{final1}
$$
b_*\(D\,\cap\,0^!_F[C_2]\)\=0^!_{F|_D}[C_D]\=0^!_{h^1/h^0(\Cone(\bsigma_D))}[\fC_D]\=[\PP(N)]^{\red}.
$$
\end{proof}

\section{Proofs of Theorems \ref{main1} and \ref{thm2}}\label{app}
Theorem \ref{main2} implies Theorem \ref{main1} immediately. However to obtain Theorem \ref{thm2} from Theorem \ref{main2} we need a surjectivity of $h^1(\bsigma)$. We introduce one criterion to check a surjectivity.

\begin{Lemma}\label{HuHuHu}
$h^1(\bsigma)$ is surjective on $D$ if and only if $h^0\(\Cone\eqref{ChainAB}\)=0$.
\end{Lemma}
\begin{proof}
By abuse of notation, we denote by $\bsigma$ the composition
$$
\bsigma\ :\ B'|_{\bN} \ \To\ B|_{\bN}\ \To\ \cO_{\bN}(D).
$$
Then $\bsigma|_D$ is surjective iff 
$$
\bsigma^*|_D\ :\ \cO_D(-D)\ \To\ B^*|_D\ \To\ B'^*|_D
$$ 
is pointwise injective iff $\ker (B^*\to A^*) \subset B^*\to B'^*$ is pointwise injective iff $h^{0}\(\Cone\eqref{ChainAB}^\vee\)=0$ at each point iff $h^0\(\Cone\eqref{ChainAB}\)=0$.
\end{proof}

Lemma \ref{HuHuHu} tells us that it is enough to check if 
$$
h^0\(\EE^\vee_{M/S}[1]\)\ \To\ h^0\(\EE_{N/M}|_M\)
$$ 
is onto for stable quasimaps. This map is described in \cite[Equation (3.7) and Example 1]{KO}. Following this description, the surjectivity is equivalent to the smoothness of $X$. Hence Theorem \ref{thm2} follows from Theorem \ref{main2}. 

\begin{Rmk}
Indeed, Theorem \ref{thm2} holds for any Calabi-Yau $3$-fold, complete intersection in the GIT quotient coming from a gauged linear sigma model studied in \cite{CFGKS}. More precisely letting $X$ be a Calabi-Yau $3$-fold, complete intersection in the GIT quotient $Y$ as the zero of a section of the bundle $V$ on $Y$, we have
\begin{align*}
\{e_{\T}\(R\pi_*f^*V\ot\t^{-1}\) & \cap [Q^{\;\varepsilon}_g(Y,d)]^{\vir}\}_{t=0} - \iota_*[Q^{\;\varepsilon}_g(X,d)]^{\vir} \\
&=(-1)^{\rank R\pi_*f^*V}p_*[\PP(N)]^{\red}\ \in\ A_0\(Q^{\;\varepsilon}_g(Y,d)\),
\end{align*}
if the defining equation of $X\subset Y$ (providing the defining section of $V$) has singularities only on the unstable locus of the GIT quotient $Y$.
\end{Rmk}

\appendix
\section{Five localisations}\label{app:A}
We review the five localised invariants introduced in \cite{JT} for a quasi-projective shceme $N$ with the symmetric obstruction theory $\EE_N\to\LL_N$, acted on by a torus $\T := \C^*$ with a compact fixed locus $N^\T$. 
We assume the POT is equivariant, but the symmetricity need not be preserved by $\T$. 

\subsubsection*{The virtual signed Euler characteristic of Ciocan-Fontanine--Kapranov and Fantechi--G\"ottsche}  The virtual Euler characteristic of $N^\T$ is defined in \cite{CK, FG} by,
$$
\int_{\, [N^\T]^{\vir}}  c\,\(\,\EE_N^\vee|_{N^\T}^{\mathrm{fix}}\)  \ \in \ \Q.
$$
The terminology {\em virtual Euler characteristic} comes from regarding $\EE_N^\vee|_{N^\T}^{\mathrm{fix}}$ to be the virtual tangent bundle of $N^\T$.
In \cite{JT}, using the virtual {\em cotangent} bundle $\EE_N|_{N^\T}^{\mathrm{fix}}$, the {\em virtual signed Euler characteristic}
$$
e_{1'} \ := \ \int_{\, [N^\T]^{\vir}}  c\,\(\,\EE_N|_{N^\T}^{\mathrm{fix}}\)  \ \in \ \Q
$$
is considered.

\subsubsection*{Graber--Pandharipande torus localisation}
The virtual cycle $[N]^{\vir}$ has a lifting in the equivariant Chow group $A^{\T}_{0}\(N\)$. Then one can localise $[N]^{\vir}$ by virtual torus localisation \cite{GP},
$$
[N]^{\vir} \ = \ \iota_*\left(\frac{[N^\T]^{\vir}}{e_{\T}\(\mathsf{N}^{\vir}_{N^\T/N}\)}\right), \ \ \  \iota\ :\ N^\T \ \into\ N.
$$
Here, $\mathsf{N}^{\vir}$ denotes the virtual normal bundle. The fixed (weight zero) part $\EE_N|_{N^\T}^{\mathrm{fix}}$ of the pullback complex $\EE_N|_{N^\T}$ is a perfect obstruction theory of $N^\T$ \cite{GP}. Thus the virtual fundamental class $[N^\T]^{\vir}$ is defined. Using the class 
$$
\frac{[N^\T]^{\vir}}{e_{\T}\(\mathsf{N}^{\vir}_{N^\T/N}\)} \ \in\ A_0\(N^\T\)[t]
$$ 
which is a polynomial in $t$ 
by degree reason,
we define an invariant
$$
e_1 \ :=\ \deg \left( \left\{\frac{[N^\T]^{\vir}}{e_{\T}\(\mathsf{N}^{\vir}_{N^\T/N}\)}\right\}_{t=0}\, \right) \ \in\ \Q.
$$

\subsubsection*{Kiem--Li cosection localisation}
The pairing with the Euler vector field defines a cosection 
$$
\sigma \ : \ \Omega_N \ \To \ \cO_N
$$
on the cotangent sheaf $\Omega_N$ of $N$ which is the obstruction sheaf $h^1\(\EE_N^\vee\) \cong h^0\(\EE_N\) \cong \Omega_N$. Then one can localise $[N]^{\vir}$ via cosection localisation \cite{KL}
$$
[N]^{\vir} \ = \ \iota_*\, [N]^{\vir}_\sigma .
$$
It defines an invariant
$$
e_2 \ :=\ \deg\, [N]^{\vir}_{\sigma} \ \in\ \Z.
$$

\subsubsection*{Behrend localisation} A weighted Euler characteristic weighted by the Behrend function $\nu_N$ \cite[Definition 1.4]{Be} gives rise to an invariant
$$
e_{2'} \ :=\ e\(N^{\T}, \nu_N|_{N^{\T}}\) \ \in\ \Z.
$$

\subsubsection*{The signed Euler characteristic} Considering $N^{\T}$ as a topological space, it produces an invariant
$$
e_{2''} \ :=\ (-1)^{\rank \(\EE_N|_{N^\T}^{\mathrm{fix}}\)} \cdot e\(N^{\T}\) \ \in\ \Z.
$$

\subsection*{Known results}
We list here some known results about the five localised invariants.

\smallskip
{\bf 1.} Kiem-Li and Behrend localised invariants are the same, $e_2=e_{2'}$, by \cite[Theorem 5.20]{Ji}.

\smallskip
{\bf 2.} When the symmetricity of $\EE_N$ is preserved by $\T$, Graber-Pandharipande and Behrend localised invariants are the same, $e_1=e_{2'}$. The following is a brief explanation. In this case $[N^{\T}]^{\vir}$ is of degree zero and
$$
e_{\T}\(\mathsf{N}^{\vir}_{N^\T/N}\) = (-1)^{\rank \(\mathsf{N}^{\vir}_{N^\T/N}\)}, \ \ \deg\,[N^{\T}]^{\vir}  = e\(N^{\T}, \nu_{N^{\T}}\).
$$
The latter comes from \cite[Theorem 4.18]{Be}. So we have 
$$
e_1\ =\ (-1)^{\rank \(\mathsf{N}^{\vir}_{N^\T/N}\)}e\(N^{\T}, \nu_{N^{\T}}\).
$$
Then $e_1=e_{2'}$ comes from \cite[Theorem 2.4]{LQ}. 

\smallskip
{\bf 3.} Combining the above results, we have $e_1=e_2$ when the symmetricity is preserved by $\T$. There is a stronger and more direct proof: by \cite[Theorem 3.5]{CKL}, we obtain an equivalence of cycles
\beq{1=2}
\frac{[N^\T]^{\vir}}{e_{\T}\(\mathsf{N}^{\vir}_{N^\T/N}\)} \ =\ [N]^{\vir}_{\sigma}.
\eeq
This induces $e_1=e_2$ immediately.
Note that \eqref{1=2} does not require $N^{\T}$ to be compact.

\smallskip
{\bf 4.}
For a $(-1)$-shifted cotangent bundle $N$ of a compact quasi-smooth derived scheme, $e_1=e_{1'}$ and $e_2=e_{2'}=e_{2''}$ are proven in \cite[Theorem 1.2]{JT}. However \cite[Examples 3.1 and 3.2]{JT} show that $e_1$ may not be equal to $e_2$ in general.

\vspace{4mm}\noindent
{\tt{j.oh@imperial.ac.uk} } \medskip

\noindent Department of Mathematics, \
Imperial College London \\
London SW7 2AZ, \
United Kingdom

\end{document}